\documentclass{article}

\usepackage[letterpaper,vmargin=1in,hmargin=1.25in]{geometry}
\usepackage{graphicx,tikz,graphics,float,fp,amsthm,cite,multicol,subcaption,caption,amssymb,mathtools,hyphenat,multirow,longtable,array}
\usetikzlibrary{automata,positioning, decorations.markings,fixedpointarithmetic,arrows}
\usepackage{longtable,comment}
\usepackage[shortlabels]{enumitem}
\usepackage{url}
\usetikzlibrary{external}

\theoremstyle{plain}
\newtheorem{theorem}{Theorem}[section]
\newtheorem{lemma}[theorem]{Lemma}

\newtheorem{corollary}[theorem]{Corollary}

\theoremstyle{definition}
\newtheorem*{example}{Example}
\newtheorem*{example*}{Example}
\newtheorem*{definition}{Definition}

\DeclareMathOperator{\NRD}{NRD}
\DeclareMathOperator{\ANRD}{ANRD}

\tikzset{
  on each segment/.style={
    decorate,
    decoration={
      show path construction,
      moveto code={},
      lineto code={
        \path [#1]
        (\tikzinputsegmentfirst) -- (\tikzinputsegmentlast);
      },
      curveto code={
        \path [#1] (\tikzinputsegmentfirst)
        .. controls
        (\tikzinputsegmentsupporta) and (\tikzinputsegmentsupportb)
        ..
        (\tikzinputsegmentlast);
      },
      closepath code={
        \path [#1]
        (\tikzinputsegmentfirst) -- (\tikzinputsegmentlast);
      },
    },
  },
  mid arrow/.style={postaction={decorate,decoration={
        markings,
        mark=at position .5 with {\arrow[#1]{stealth}}
      }}},
}

\tikzset{every loop/.style={min distance=10mm,looseness=10}} 
\makeatletter 

\DeclarePairedDelimiter{\ceil}{\lceil}{\rceil}

\usepackage{tikz}
\tikzstyle{vert}=[shape=circle,draw=black,fill=white, inner sep=.5mm]
\usetikzlibrary{decorations.pathreplacing,decorations.markings,arrows,calc,math,arrows.meta}
\newcolumntype{C}{>{$}c<{$}}
\newcolumntype{M}{>{$}p{10cm}<{$}}

\title{Localization and metric dimension for \\ families of highly structured digraphs}

\author{
Robert F.\ Bailey\thanks{ School of Science and the Environment (Mathematics), Grenfell Campus, Memorial University, Corner Brook, NL, A2H 6P9, Canada. E-mail: \texttt{robert.bailey@mun.ca}, \texttt{bep275@mun.ca} }
\and Brittany Pittman\footnotemark[1]
}

\date{\today}

\begin{document}

\maketitle


\begin{abstract}
We investigate metric dimension and the localization game for several families of directed analogues of strongly regular graphs and their generalizations, adapting a probabilistic method of Babai (1980) for bounding the size of resolving sets in undirected strongly regular graphs.  We derive upper bounds on the localization number and metric dimension depending on the order of the graph and the maximum number of common out-neighbours for a pair of vertices.  We consider normally regular digraphs, so-called ``ordinary graphs'', classes of Deza digraphs, divisible design digraphs, nearly doubly regular tournaments, and certain doubly regular team tournaments.  In particular, for asymmetric normally regular digraphs on $n$ vertices, we show that these invariants are bounded above by $O(\sqrt{n} \log n)$, and improve this to $O(\log n)$ for a class of doubly regular team tournaments.
\end{abstract}


\section{Introduction}

Pursuit-evasion games are combinatorial models used to study the detection, containment, or capture of an adversary moving within a network, typically represented as a graph. In these games, a group of pursuers, such as a set of cops, attempts to track down and capture an adversary, such as a robber, who is moving through a graph. These games are played in discrete time steps, and players move in alternating rounds. Each player is subject to movement rules that depend on the game being studied. One of the most well-known of these models is \textit{Cops and Robber}, defined by Nowakowski and Winkler \cite{NW} and, independently, by Quilliot \cite{Qcops}. In this game, a set of cops and a robber move to adjacent vertices in alternate rounds. The cops aim to capture the robber in a finite number of rounds, and the robber aims to avoid capture indefinitely. The main question of this game is to determine the minimum number of cops required to guarantee the robber's eventual capture. Many variants of this game have been studied, and each introduce new rules or restrictions for the players. For instance, we can limit the visibility of the robber by only allowing the cop to see the robber when they are within a specified distance from the robber (see \cite{LMCR}).  An overview of pursuit-evasion games is given in~\cite{Bonato}.


\subsection{The localization game}
In this paper, we focus on a variant called the \textit{localization game} in which the robber is invisible to the cops. It was introduced independently by Carraher et al.\ in \cite{Carraher} and by Seager in \cite{seager1,seager2}. In the localization game, one player controls a set of \textit{cops}, and the other controls a single \textit{robber}. It is played in discrete time-steps, which we call \textit{rounds}. To start the game, the robber chooses an initial vertex to occupy. In each subsequent round, the cops move, followed by the robber. The cops can choose any set of vertices to occupy, but the robber can only move to a neighbouring vertex (or remain on their current position).

After each time the robber moves, each cop sends out a \textit{cop probe} and that cop receives their distance to the robber. If the cops occupy vertices $u_1, u_2, \ldots, u_k,$ then each cope probe gives the distance $d_i$ from $u_i$ to the vertex occupied by the robber. This gives a \textit{distance vector}, written $(d_1, d_2, \ldots, d_k).$ The robber is \textit{located} if, in any round, there is a unique vertex corresponding to the distance vector received by the cops. If there is a strategy to locate the robber in a finite number of moves, the cops win. The minimum number of cops required to locate the robber in a finite number of rounds (when the robber is playing optimally) is the \textit{localization number}, written $\zeta(G)$ for a graph $G$.

In \cite{BCMP}, the localization game was extended to oriented graphs. The game is played analogously on oriented graphs, with the robber moving along directed edges and the distance probes measuring the length of the shortest directed path from the cop to the robber. 

This naturally leads to a broader question: how does the localization game behave on more general directed graphs, including those with digons (two arcs in opposite directions between a single pair of vertices)? In this more general setting, there are added possibilities for both the robber's movement and the cops' distance measurements, adding complexity to the game.

In a symmetric digraph, for any arc $(u,v)$, there is also an arc $(v,u)$. In a symmetric digraph, each arc has a corresponding arc in the opposite direction, so the digraph behaves similarly to an undirected graph in terms of movement and distance. We begin with a result which shows that replacing each undirected edge with a digon does not alter the localization number.

\begin{theorem}
If $D$ is a symmetric digraph and $G$ is the underlying undirected graph, then $\zeta(D)=\zeta(G).$
\end{theorem}

This result follows from the fact that, in a symmetric digraph, the distance from any vertex to another is the same as in the underlying undirected graph. Consequently, the distance vectors produced by cop probes are identical in both settings when the cops and robber positions are the same in each graph, and thus the same number of cops are required to uniquely determine the robber's position. Since the localization game proceeds identically on symmetric digraphs and their undirected counterparts, our attention will be restricted to digraphs that are not symmetric. 


\subsection{Metric dimension}

A {\em resolving set} in an undirected graph $G$ is a subset $R$ of the vertex set chosen so that for any pair of vertices $u,v$ or $G$, there exists some $w\in R$ with $d(u,w)\neq d(v,w)$.  The smallest size of a resolving set is called the {\em metric dimension} of $G$, which we denote by $\dim(G)$.  These concepts were introduced to graph theory in the 1970s by Harary and Melter~\cite{HM76} and, independently, Slater~\cite{Slater75}; however, in more general metric spaces, the concept can be found in the literature much earlier (see~\cite{Blumenthal53}).  For further details, the reader is referred to the surveys~\cite{bsmd,Tillquist}.  We note that the metric dimension of a graph is an upper bound on its localization number, as if the cops begin by occupying a resolving set they will win the game immediately.

Metric dimension was first defined for oriented graphs by Chartrand, Raines, and Zhang in 2000~\cite{chartrand-raines-zhang-2000}. They define a resolving set to be a set $R$ such that for each pair of vertices $u,v$ not in $R$, there exists a vertex $w$ in $R$ such that the distance from $u$ to $w$ differs from the distance from $v$ to $w$. They require that $d(u,w)$ and $d(v,w)$ are defined. As in the undirected case, the metric dimension of a directed graph $D$, denoted by $\dim(D)$, is the cardinality of the smallest resolving set.  Unlike the undirected case, it is possible that a resolving set may not exist, in which case $\dim(D)$ is undefined.

The metric dimension of digraphs was later studied by Bensmail, Mc~Inerney, and Nisse in \cite{ben}. They instead define a resolving set to be a set $R$ such that for each pair of vertices $u,v$ not in $R$, there exists a vertex $w$ in $R$ such that the distance from $w$ to $u$ differs from the distance from $w$ to $v$. That is, they consider the distances from $R$ to vertices not in $R$, rather than the distance from vertices not in $R$ to the vertices in $R$. In their definition, they allow for undefined distances, and hence, the metric dimension is defined for any digraph under their definition. Note that in the setting of strongly connected digraphs, the two definitions for resolving sets and metric dimension are equivalent.

In this paper, we use the notion of metric dimension introduced in \cite{ben}. Unlike the resolving sets defined in \cite{chartrand-raines-zhang-2000}, under this definition a resolving set is not necessarily an optimal placement for the cops in the localization game.
However, there always exists a set of size $\dim(D)$ that is a resolving set under the original definition from \cite{chartrand-raines-zhang-2000} extended to permit undefined distances.
In general, given a digraph $D$, we can find a minimum resolving set $\widetilde{R}$ in the digraph $\widetilde{D}$ obtained by reversing all arcs of $D$. In $D$, the set $\widetilde{R}$ is a set of size $\dim(D)$ such that for each pair of vertices $u,v$ not in $R$, there exists a vertex $w$ in $R$ such that the distance from $u$ to $w$ differs from the distance from $v$ to $w$. That is, in $D$, $\widetilde{R}$ is a resolving set as defined in \cite{chartrand-raines-zhang-2000} extended to the weakly connected setting. Thus this set is an optimal placement of $\dim(D)$-many cops, so (as in the undirected setting) the localization number of a digraph is bounded above by its metric dimension.

In undirected graphs, several variants of resolving sets have been defined. One such variant is a distinguishing set, which is a set $S$ of vertices that distinguish any pair of vertices $u$ and $v$ not in $S$. A set $S$ distinguishes $u$ and $v$ if $N(u)\cap S \neq N(v)\cap S$. Any distinguishing set is also a resolving set. We extend this definition to directed graphs as follows. Given a digraph $D$ and a set $S\subseteq V(D)$, we say that $S$ \textit{distinguishes} distinct vertices $u, v\in V(D)$ if either $u\in S$, $v\in S$, or $N^-(u)\cap S \neq N^-(v)\cap S.$ We say that $S$ is a \textit{distinguishing set} if $S$ distinguishes each pair of distinct vertices. Any distinguishing set is also a resolving set, and it follows that if $D$ has a distinguishing set $S$, then $|S|\geq \dim (D)\geq \zeta(D)$.  The terminology ``distinguishing set'' is due to Babai~\cite{LB}; we will use this, but not the term ``distinguishing number'' (to refer to the least size of such a set) which has become widely used with a different meaning, introduced by Albertson and Collins~\cite{AlbCol}.


\subsection{Strongly regular graphs and doubly regular tournaments}

A {\em strongly regular graph} with parameters $(n,k,\lambda,\mu)$ is an undirected graph with $n$ vertices, which is $k$-regular, and where there exist non-negative integers $\lambda$ and $\mu$ such that every pair of adjacent vertices have exactly $\lambda$ common neighbours and every pair of non-adjacent vertices have exactly $\mu$ common neighbours. Equivalently, in such a graph, there are precisely $\lambda$ paths of length two between each pair of adjacent vertices and $\mu$ such paths between each pair of non-adjacent vertices.

The metric dimension of strongly regular graphs was first discussed by Bailey and Cameron in 2011~\cite[{\S}3.7]{bsmd}, where it was observed that work of Babai from 1980~\cite{LB} (which uses distinguishing sets to study the complexity of the graph isomorphism problem for strongly regular graphs) yields bounds on the metric dimension of such graphs; in particular, for any strongly regular graph $G$ with $n$ vertices (other than a complete multipartite graph), we have $\dim(G) = O(\sqrt{n}\log n)$.  In this paper, our goal is to adapt Babai's techniques to the setting of directed graphs.

Our motivating examples are the (undirected) Paley graphs and their directed analogues, the Paley tournaments.  For a prime power $q\equiv 1$~mod~$4$, define a graph $P_q$ whose vertices are the elements of the finite field $\mathbb{F}_q$, with two being adjacent if their difference is a perfect square.  It is well-known that this graph is strongly regular with parameters $(q,\ (q-1)/2,\ (q-5)/4,\ (q-1)/4)$, and is isomorphic to its complement.  A consequence of Babai's work (later observed independently by Fijav\v{z} and Mohar~\cite{FijavzMohar04} when $q$ is prime) is that $\zeta(P_q)\leq\dim(P_q)=O(\log q)$.  

In the case where $q\equiv 3$~mod~$4$, an element $a\in\mathbb{F}_q$ is a perfect square if and only if $-a$ is not, and so the Paley construction yields an orientation of the complete graph $K_q$ rather than a pair of self-complementary graphs.  In the {\em Paley tournament}, which we denote by $\overrightarrow{P_q}$, the vertex set is still $\mathbb{F}_q$, and there is an arc from $x$ to $y$ when $y-x$ is a perfect square.  This is an example of a {\em doubly regular tournament} (DRT) of order $q$ (see~\cite{ReidBrown72}), i.e.\ a directed graph with $q$ vertices, each with in-degree and out-degree $(q-1)/2$, and where any pair of distinct vertices have exactly $(q-3)/4$ common out-neighbours and $(q-3)/4$ common in-neighbours.  In \cite[Corollary 18]{BCMP} it was shown that $\zeta(\overrightarrow{P_q})\leq\dim(\overrightarrow{P_q})=O(\log q)$, just as with the undirected case.

We note that these bounds hold for any strongly regular graph with the parameters of the Paley graph, and for any doubly regular tournament.  However, since Babai's work for undirected strongly regular graphs is much more general, it is natural to investigate directed analogues.


\section{Families of highly structured digraphs}
When defining a directed analogue of strongly regular graphs, the equivalence between counting the number of common neighbours of a pair of vertices and the counting number of $2$-paths between them no longer exists.  Possibly the most well-known analogue is the {\em directed strongly regular graph} (DSRG), introduced by Duval in 1988~\cite{Duval}: in such a graph, each vertex has in-degree and out-degree equal to $k$, and for any pair of vertices $x$ and $y$, the number of directed $2$-paths from $x$ to $y$ is either $t$, $\lambda$ or $\mu$, depending (respectively) on whether $x=y$, if there is an arc from $x$ to $y$, or otherwise.  However, as our techniques rely on knowledge of the number of common neighbours of $x$ and $y$ rather than directed $2$-paths, we do not discuss DRSGs further.  Rather, we will focus on analogues defined in terms of the former.

\subsection{A useful lemma}
To obtain our main results, we adapt the proof technique used by Babai in \cite{LB} for undirected strongly regular graphs to the directed setting. Specifically, in each case we demonstrate the existence of a distinguishing set of the appropriate size. We begin by proving the following lemma, which serves as a directed analogue of \cite[Lemma 3.2]{LB}, and whose proof we follow closely (see also~\cite[Theorem 2.4]{incidence}).

\begin{lemma}\label{distset}
Let $D$ be a digraph of order $n$, and suppose that there exists some $c>0$ such that $|N^-(u) \triangle N^-(v)|\geq c$ for any pair of distinct vertices $u$ and $v$. Then, provided that $c>2\log n$, $D$ has a distinguishing set of cardinality $\ceil{2n\log n/c}$.
\end{lemma}

\begin{proof}
    Let $s$ be an integer and choose a subset $S\subseteq V(D)$ of cardinality $s$ at random. For distinct vertices $u$ and $v$, we let $A(u,v)$ denote the event that $S$ does not distinguish $u$ and $v$. We let $P(u,v)$ denote the probability of the event $A(u,v).$ If $S$ does not distinguish $u$ and $v$, then both $u$ and $v$ are not in $S$ and $N^-(u)\cap S =N^-(v)\cap S.$ That is, if $S$ does not distinguish $u$ and $v$, then the $s$ vertices in $S$ must have been chosen from the (at most) $n-c-2$ vertices in $V(D)\setminus (\{u,v\}\cup (N^-(u) \triangle N^-(v))).$
    
    It follows that 
    \[ \left. P(u,v)\leq \binom{n-c-2}{s}\middle/ \binom{n}{s}<\binom{n-c}{s} \middle/\binom{n}{s}. \right. \]
    Let $N$ be the number of pairs $u,v$ such that $A(u,v)$ holds. Then
    \[ \left. E(N)=\sum_{\{u,v\}} P(u,v)\leq \binom{n}{2}\binom{n-c}{s}\middle/\binom{n}{s}. \right. \]
Now let $s=\ceil{2n\log n/c}$.  Since $c>2\log n$, it follows that $s\leq n$.  Also, we observe that
$$ \frac{sc}{n} > 2\log n - \log 2$$
from which it follows that
\[ {n\choose 2} < \frac{n^2}{2} < \exp(c/n)^s < \left( 1 + \frac{c}{n} + \frac{c^2}{n^2} \right)^s < \prod_{i=0}^{s-1} \left( 1+ \frac{c}{n-c-i} \right) = {n \choose s} \biggm/ {n-c \choose s}. \]
(Note that $e^x< 1+x+x^2$ for $0<x<1$; the final inequality can be proved by induction on $s$.)  

As a result, we can conclude that $E(N)<1$ and hence $\textrm{Prob}(N=0)>0.$ Since $S$ is a distinguishing set if and only if $N=0$, it follows that a distinguishing set of cardinality $s$ exists.
\end{proof}


\subsection{Normally regular digraphs}

In this section, we consider {\em normally regular digraphs} as defined by J{\o}rgensen in a 2015 paper~\cite{nrd} (which developed from a much earlier unpublished manuscript, dating from 1994~\cite{nrd-old}).  These are as follows.

\begin{definition}
Let $D$ be a digraph with $n$ vertices.  We say that $D$ is a {\em normally regular digraph} if there exist integers $k$, $\lambda$, $\mu$ such that each vertex has in-degree and out-degree $k$, and for every pair of vertices $x$ and $y$ the number of common out-neighbours is $\lambda$ if there is a single arc between $x$ and $y$, $\mu$ if there is no such arc, or $2\lambda-\mu$ if there are arcs in both directions.
\end{definition}

We will refer to such a digraph $D$ as an $\NRD(n,k,\lambda,\mu)$.  The reason for the requirement for the vertices of a digon to have $2\lambda-\mu$ common out-neighbours is to allow the adjacency matrix $A$ of $D$ to satisfy the matrix equation
\[ AA^t = kI + \lambda(A+A^t) + \mu(J-I-A-A^t) \]
(where $J$ denotes the all-ones matrix), which is analogous to the well-known one for strongly regular graphs (see~\cite[Theorem 9.1.2]{BH}).

\begin{example}
The Cayley digraph of the quaternion group $Q_8=\{\pm 1, \pm i, \pm j, \pm k\}$ with connection set $\{i,j,k\}$, depicted in Figure~\ref{q8}, is an $\NRD(8,3,1,0)$ (see \cite[Example 2]{nrd}).  It is straightforward to verify that $\{1,i,j\}$ is a resolving set of least size, and thus that its metric dimension is~$3$.
\end{example}

\begin{figure}[h] 
\centering 
\begin{tikzpicture}

 \def \n {8}
    \node[draw, circle, fill=black,label=above:{$i$}] (3) at ({360/\n * (3 - 1)}:4.4cm) {};
    \node[draw, circle, fill=black,label=above right:{$j$}] (2) at ({360/\n * (2 - 1)}:4.4cm) {};
    \node[draw, circle, fill=black,label=right:{$1$}] (1) at ({360/\n * (1 - 1)}:4.4cm) {};
    \node[draw, circle, fill=black,label=below right:{$k$}] (8) at ({360/\n * (8 - 1)}:4.4cm) {};
    \node[draw, circle, fill=black,label=below:{$-i$}] (7) at ({360/\n * (7 - 1)}:4.4cm) {};
    \node[draw, circle, fill=black,label=below left:{$-j$}] (6) at ({360/\n * (6 - 1)}:4.4cm) {};
    \node[draw, circle, fill=black,label=left:{$-1$}] (5) at ({360/\n * (5 - 1)}:4.4cm) {};
    \node[draw, circle, fill=black,label=above left:{$-k$}] (4) at ({360/\n * (4 - 1)}:4.4cm) {};
\path [fixed point arithmetic, thick, draw=black,
         decorate,
         postaction={
    on each segment={decoration={markings,mark=at position 0.5 with {\arrow[line width =3pt]{stealth}}},decorate}}]
(1) -- (8)
(1) -- (2)
(1) -- (3)
(2) -- (4)
(2) -- (5)
(2) -- (3)
(3) -- (8)
(3) -- (6)
(3) -- (5)
(4) -- (3)
(4) -- (1)
(4) -- (6)
(5) -- (6)
(5) -- (4)
(5) -- (7)
(6) -- (1)
(6) -- (7)
(6) -- (8)
(7) -- (1)
(7) -- (2)
(7) -- (4)
(8) -- (5)
(8) -- (7)
(8) -- (2)
;

\end{tikzpicture}

\caption{This Cayley digraph of the quaternion group $Q_8$ is an $\NRD(8,3,1,0)$.}
\label{q8}
\end{figure}
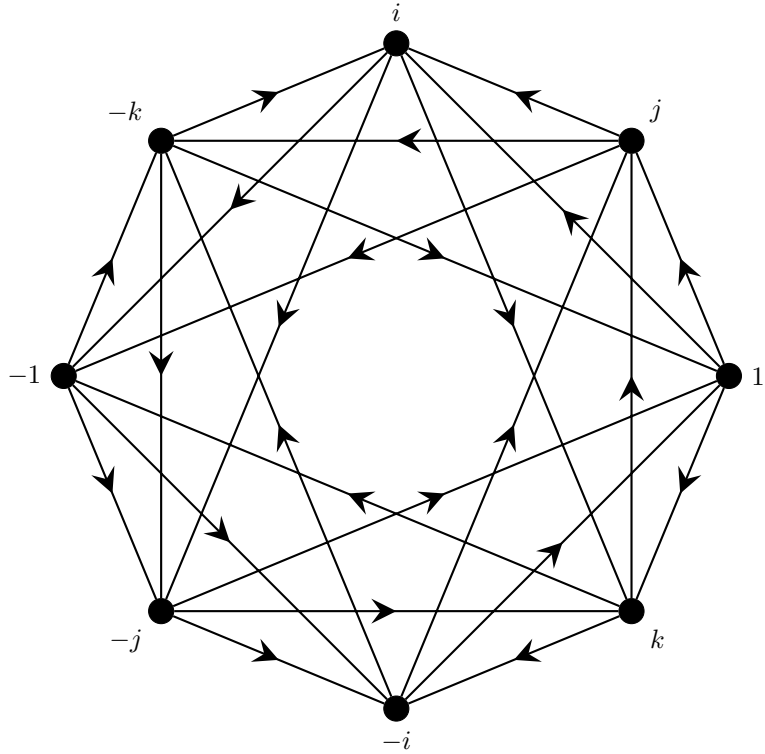

A normally regular digraph is {\em asymmetric} if it contains no digons; we denote such digraphs by $\ANRD(n,k,\lambda,\mu)$.  An ANRD where $\lambda=\mu$ is a {\em doubly regular asymmetric digraph} (DRAD), as studied by Ito~\cite{Ito88}; many of the known constructions of NRDs are in fact DRADs (see~\cite{nrd}).  Also, a DRAD with $k=\frac{n-1}{2}$ and $\lambda=\mu=\frac{n-3}{4}$ is precisely a doubly regular tournament as discussed above.


Using the definition of NRDs, it is straightforward to obtain a lower bound on $|N^-(u) \triangle N^-(v)|$ in terms of the parameters, and thus to apply Lemma~\ref{distset} to obtain bounds on the metric dimension and localization number.  We have the following result.

\begin{theorem} \label{NRD-easy}
Let $D$ be an $\NRD(n,k,\lambda,\mu)$, and let $t=\min\{\lambda,\mu,2\lambda-\mu\}$.  Then, provided that $k-t>\log n$, we have $\zeta(D)\leq \dim(D) \leq  \ceil{n\log n/(k-t)}$.
\end{theorem}

\begin{proof}
Since $D$ is an $\NRD(n, k, \lambda, \mu)$, we know that for any vertices $u, v$, $|N^-(u) \triangle N^-(v)| = 2k-2\lambda$ (if there is one arc between $u$ and $v$), $2k-2\mu$ (if there is no such arc) or $2k-2(2\lambda-\mu)$ (if there are arcs in both directions).  So we have $|N^-(u) \triangle N^-(v)|\geq 2k-2t$, and can therefore apply Lemma~\ref{distset} to complete the proof.
\end{proof}

In the case of ANRDs, where there are no digons, the third possibility for $|N^-(u) \triangle N^-(v)|$ does not arise, so it suffices to let $t=\max\{\lambda,\mu\}$; in the case of DRADs, we have $\lambda=\mu$ and thus the bound simplifies to $\zeta(D)\leq \dim(D) \leq  \ceil{n\log n/(k-\lambda)}$. In the case of DRTs, this simplifies further to $\ceil{4n\log n / (n+1)} < 4\log n$, concurring with the $O(\log n)$ bound for the localization number of DRTs given in~\cite[Corollary 17]{BCMP}.

However, we can obtain a stronger result, again by adapting the techniques from~\cite{LB} to obtain estimates on $|N^-(u) \triangle N^-(v)|$ solely in terms of the number of vertices (provided that $\lambda,\mu>0$).  We do this in Lemma~\ref{sd} below; the proof of this requires the following two lemmas from the literature.

\begin{lemma}{\em\cite{nrd}}\label{mindegree}
    If $D$ is a $\NRD(n, k, \lambda, \mu)$ with $n\geq 2k+1$ or an $\ANRD(n, k, \lambda, \mu)$ then $$2k\lambda + (n-2k-1)\mu=k^2-k.$$
\end{lemma}

In particular, when $\mu, \lambda >0$ it follows that $k\geq \sqrt{n-1}$ as $k^2 \geq k^2-k = 2k\lambda + (n-2k-1)\mu > 2k + (n-2k-1) = n-1$ when $\lambda, \mu \geq 1$.

\begin{lemma}{\em\cite{LB}}\label{hypergraphs}
    Let $1\leq d <r<n$ be integers, $|V|=n$, and $H=(V, \textbf{F})$ be a nonempty regular $r$-uniform hypergraph such that $|E\cap F|\geq d$ for any $E, F\in \textbf{F}$. Then $r^2>nd$.
\end{lemma}

The following result establishes a lower bound on the size of the symmetric difference of the in-neighbourhoods of any two vertices.

\begin{lemma}\label{sd}
    Let $D$ be a $\ANRD(n, k, \lambda, \mu)$. If $\lambda, \mu > 0$, then $$|N^-(u) \triangle N^-(v)| \geq \sqrt{n}-1$$ for any pair of distinct vertices $u$ and $v$.
\end{lemma}

\begin{proof}
    Let $\lambda, \mu$ be positive integers. Since $D$ is an $\ANRD(n, k, \lambda, \mu)$, we know that $\sqrt{n-1}\leq k \leq (n-1)/2.$ We also know that for any vertices $u, v$, $|N^-(u) \triangle N^-(v)| = 2k-2\lambda$ (if there is one arc between the vertices) or $2k-2\mu$ (if there is no arc connecting the vertices).

   First, let $u,v$ be distinct vertices such that there is no arc between $u$ and $v$, and choose a vertex $w$ such that there are arcs from $u$ to $w$ and from $v$ to $w$ (such a vertex exists since $\mu>0$). We have that $|N^-(u) \setminus N^-(v)|=k-\mu$ and $|N^-(u) \setminus N^-(w)|=|N^-(w) \setminus N^-(v)|=k-\lambda$.  We can show that 
	$$N^-(u)\setminus N^-(v) \subseteq (N^-(u)\setminus N^-(w))\cup (N^-(w) \setminus N^-(v)).$$
	It follows that $k-\mu \leq 2(k-\lambda)$.

    Next, consider vertices $u,v$ such that there is an arc from $u$ to $v$, and choose a vertex $w$ such that there are arcs from $u$ to $w$ and from $v$ to $w$ (such a vertex exists since $\lambda>0$).  A similar argument to the above shows that $k-\lambda \leq 2(k-\mu)$.
		
		Now let $d=\min \{\lambda, \mu\}.$ From Lemma~\ref{mindegree}, we observe that $k\geq 2\lambda$.  Also, since $k \leq (n-1)/2$, it follows that $n\geq 4\lambda \geq 4d$ and thus $d<\frac{n}{4}.$  Moreover, for any pair of distinct vertices $u$ and $v$, if $u$ and $v$ are adjacent we have
		$$|N^-(u)\triangle N^-(v)|= |N^-(u)|+|N^-(v)|-2|N^-(u)\cap N^-(v)|)= 2(k-\lambda)\geq k-\mu,$$  
	while if $u$ and $v$ are non-adjacent we have
		$$|N^-(u)\triangle N^-(v)|= |N^-(u)|+|N^-(v)|-2|N^-(u)\cap N^-(v)|)= 2(k-\mu)\geq k-\lambda.$$
		In both cases, we have $|N^-(u)\triangle N^-(v)| \geq k-d$.  We want to show that $k-d>\sqrt{n}-1$.

    If $d=0,$ then $k-d=k\geq \sqrt{n-1}>\sqrt{n}-1$, as required. Otherwise, suppose that $d>0$. We can apply Lemma \ref{hypergraphs} to the non-empty $k$-uniform hypergraph $\displaystyle{\left( V(D),\, \{N^-(v): v\in V(D)\} \right)}$ to see that $k>\sqrt{nd}$. It follows that $$k-d>\sqrt{nd}-d.$$  We want to show that $\sqrt{nd}-d>\sqrt{n}-1,$ or equivalently, that $$\sqrt{n}(\sqrt{d}-1)\geq d-1.$$ Since $n\geq 4d,$ it suffices to show that $\sqrt{4d}(\sqrt{d}-1)\geq d-1.$ Simplifying this expression gives $$d+1\geq 2\sqrt{d}.$$ After squaring both sides and rearranging, we get $$(d-1)^2\geq 0,$$ which is true for any $d\geq 1.$ Hence, $$|N^-(u)\triangle N^-(v)| >k-d >\sqrt{n}-1$$ as required. 
\end{proof}

We can now prove the following bound for asymmetric normally regular digraphs.

\begin{theorem}\label{ANRD} 
Let $D$ be an $\ANRD(n, k, \lambda, \mu)$ such that $\lambda, \mu >0$ and $n\geq 108$.
Then $D$ has a distinguishing set of cardinality $\ceil{ 2n\log n / (\sqrt{n}-1) }$ and hence $\zeta(D)\leq \dim(D) \leq \ceil{ 2n\log n / (\sqrt{n}-1) }$.
\end{theorem}

\begin{proof}
    By Lemma \ref{sd}, if $D$ is an $\ANRD(n, k, \lambda, \mu)$ such that $\lambda, \mu>0$, then for any pair of distinct vertices $u$ and $v$ we have $|N^-(u) \triangle N^-(v)| \geq \sqrt{n}-1$. Since $n\geq 108$ we have $\sqrt{n}-1>2\log n$, so we can apply Lemma \ref{distset}, from which it follows that $D$ has a distinguishing set $S$ of cardinality $\ceil{ 2n\log n / (\sqrt{n}-1) }$. Since any distinguishing set is also a resolving set, $\zeta(D)\leq \dim(D) \leq \ceil{ 2n\log n / (\sqrt{n}-1) }$.
\end{proof}

We note that Theorem~\ref{ANRD} gives an upper bound on $\zeta(D)$ and $\dim(D)$ of $O(\sqrt{n}\log n)$ for an ANRD $D$ with $n$ vertices.

In the special case of doubly regular asymmetric digraphs, the following bound can be derived as a corollary to Theorem~\ref{ANRD}.

\begin{corollary} \label{DRAD}
    Let $D$ be a doubly regular asymmetric digraph with parameters $n,k$ and $\lambda$ such that $n\geq 108$ and $\lambda>0$.
		Then $\zeta(D) \leq \dim(D)\leq \ceil{2n\log n / (\sqrt{n}-1)}$.
\end{corollary}

We conclude this section by mentioning a related class of digraphs, the rather unfortunately named ``ordinary graphs'' introduced in a 2004 paper of Fossorier, Je{\v{z}}ek, Nation and Pogel~\cite{ordinary} (see also~\cite{Kalk}).  These are defined as follows: a digraph $D$ is an \emph{ordinary graph} of type $\langle n,k,a,b,c \rangle$ if $D$ has $n$ vertices, is $k$-regular, and the number of common in-neighbours and out-neighbours between any pair of distinct vertices $x$ and $y$ is $a$ if there is no arc between $x$ and $y$, $b$ if there is a single arc between $x$ and $y$, and $c$ if there is an arc in both directions.  Using the same approach as with Theorem~\ref{NRD-easy}, we have the following result.

\begin{theorem} \label{ordinary}
    Let $D$ be an ordinary graph of type $\langle n,k,a,b,c \rangle$. Let $d=\max\{a,b,c\}.$ If $k-d >\log(n)$, then $\zeta(D) \leq\dim(D) \leq \ceil{n\log n/(k-d)}.$
\end{theorem}


\newpage

\subsection{Deza digraphs and divisible design digraphs}
In this section, we consider some alternative generalizations of DRADs, and thus of DRTs.

In \cite{WF1}, Wang and Feng defined a $(n, k, a, b)$-\emph{Deza digraph} to be a $k$-regular digraph on $n$ vertices such that any two vertices have either $a$ or $b$ common out-neighbours, where $a\leq b$. 
See Figure~\ref{ddt2} for a $(6, 2, 0, 1)$-Deza digraph: we note that this is not an NRD.

\begin{figure}[h] 
\centering \begin{tikzpicture}[scale=1]

 \def \n {6}
\foreach \s in {1,...,6}
  \node[draw,scale=0.85, circle, fill=black,label=above:{}] (\s) at ({360/\n * (\s - 1)}:4.4cm) {};

\path [fixed point arithmetic, thick, draw=black,
         decorate,
         postaction={
    on each segment={decoration={markings,mark=at position 0.55 with {\arrow[line width =3pt]{stealth}}},decorate}}]
(1) -- (2)
(1) -- (3)
(2) -- (3)
(2) -- (4)
(3) -- (4)
(3) -- (5)
(4) -- (5)
(4) -- (6)
(5) -- (1)
(5) -- (6)
(6) -- (1)
(6) -- (2)
;

\end{tikzpicture}

\caption{A $(6, 2, 0, 1)$-Deza digraph}
\label{ddt2}
\end{figure}
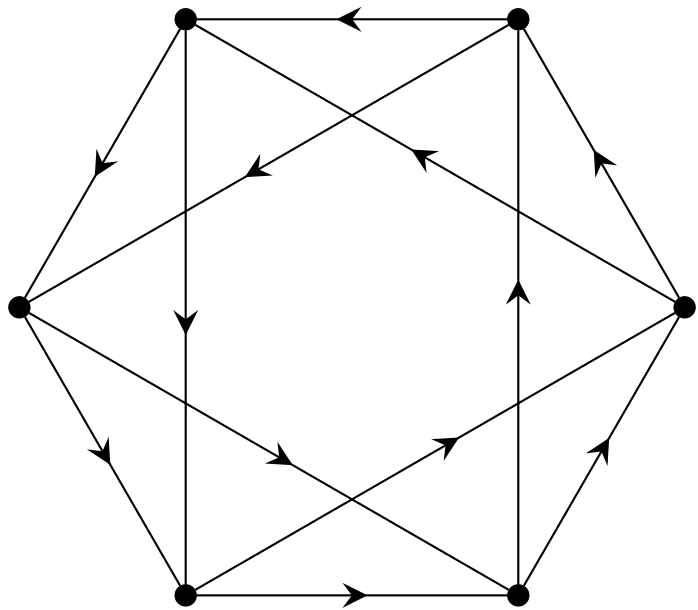

The following upper bound on the metric dimension and localization number of Deza digraphs is an immediate consequence of Lemma~\ref{distset}.  The proof is almost identical to that of Theorem~\ref{NRD-easy}.

\begin{theorem}
    Let $D$ be a $(n, k, a, b)$-Deza digraph. If $k-b >\log(n)$ then $\zeta(D)\leq \dim(D) \leq \ceil{n\log n/(k-b)}$.
\end{theorem}

We also consider a generalization of DRADs known as divisible design digraphs. We say that a $k$-regular asymmetric digraph $D$ is a {\em divisible design digraph} with parameters $(v,k,\lambda_1, \lambda_2, m,n)$ if $D$ has $v$ vertices which can be partitioned into $m$ classes of size $n$ such that the number of common in-neighbours and common out-neighbours of any two vertices in the same class is $\lambda_1$, and is $\lambda_2$ for any two vertices from different classes. These digraphs were defined by Crnkovi\'c and Kharaghani in 2015 in \cite{DDD}; additional constructions can be found in~\cite{DDCD}.

As a direct result of Lemma~\ref{distset}, we obtain the following bound.

\begin{theorem}
    Let $D$ be a $(v, k, \lambda_1, \lambda_2, m, n)$-divisible design digraph such that $\lambda_2>0$. Let $c=\max\{\lambda_1, \lambda_2\}.$ If $k-c >\log(v)$ then $\zeta(D)\leq\dim(D) \leq \ceil{v\log v/(k-c)}.$
\end{theorem}


\subsection{Tournaments and team tournaments}

Finally, we consider two classes of objects analogous to DRTs. We note that a regular tournament is doubly regular if the in-neighbourhoods and out-neighbourhoods of each vertex induce a regular tournament.  A \emph{near-regular} tournament is a tournament on an even number of vertices such that the out-degree of each vertex is either $n/2$ or $(n-2)/2$.  A \emph{nearly-doubly-regular tournament} of order $n$, denoted by $\mathrm{NDR}(n)$, is a regular tournament such that the in-neighbourhoods and out-neighbourhoods of each vertex induce a near-regular tournament. For any $\mathrm{NDR}(n)$, we must have that $n\equiv 1$~mod~$4$. See Figure \ref{ndrt} for a nearly doubly regular tournament of order $9$: in this example, any pair of vertices have either $2$ or $3$ common in-neighbours and common out-neighbours.

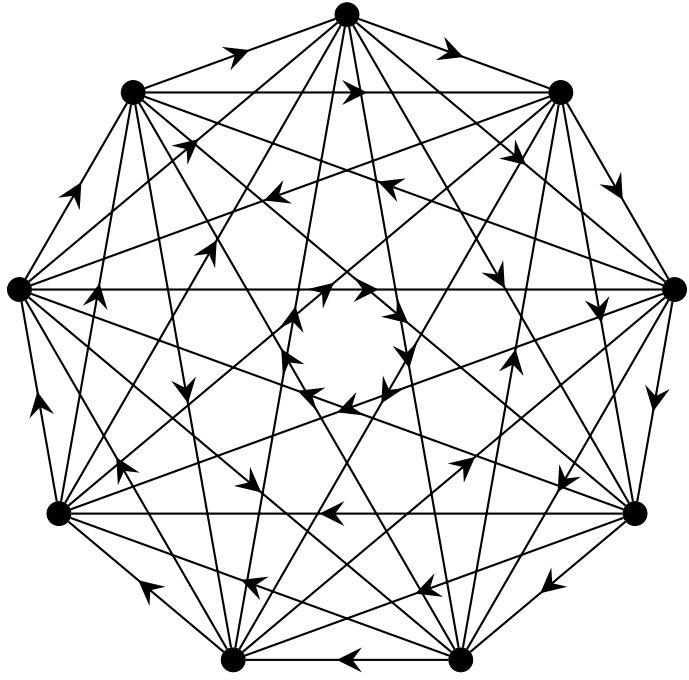
\begin{figure}[h] 
\centering \begin{tikzpicture}[scale=1, rotate=10]

 \def \n {9}
\foreach \s in {1,...,9}
  \node[draw,scale=0.95, circle, fill=black,label=above:{}] (\s) at ({360/\n * (\s - 1)}:4.4cm) {};

\path [fixed point arithmetic, thick, draw=black,
         decorate,
         postaction={
    on each segment={decoration={markings,mark=at position 0.55 with {\arrow[line width =3pt]{stealth}}},decorate}}]
(1) -- (4)
(1) -- (6)
(1) -- (8)
(1) -- (9)
(2) -- (5)
(2) -- (1)
(2) -- (7)
(2) -- (9)
(3) -- (1)
(3) -- (2)
(3) -- (8)
(3) -- (9)
(4) -- (2)
(4) -- (3)
(4) -- (7)
(4) -- (9)
(5) -- (1)
(5) -- (3)
(5) -- (4)
(5) -- (8)
(6) -- (2)
(6) -- (3)
(6) -- (4)
(6) -- (5)
(7) -- (1)
(7) -- (3)
(7) -- (5)
(7) -- (6)
(8) -- (2)
(8) -- (4)
(8) -- (6)
(8) -- (7)
(9) -- (5)
(9) -- (6)
(9) -- (7)
(9) -- (8)
;

\end{tikzpicture}

\caption{A nearly doubly regular tournament of order $9$.}
\label{ndrt}
\end{figure}

\begin{theorem}
    Let $D$ be a nearly doubly regular tournament of order $n$. If $n\geq 5$ then $\zeta(D) \leq\dim(D) \leq \ceil{4n\log n/(n-1)} = O(\log n).$
\end{theorem}

\begin{proof}
Let $D$ be a nearly doubly regular tournament of order $n\geq 5$. Consider two vertices, $x$ and $y.$ Without loss of generality, suppose there is an arc from $x$ to $y.$ Since $D$ is an $\mathrm{NDR}(n)$, the in-neighbourhood of $y$ is a near-regular tournament containing $x$. The out-degree of $x$ within this subtournament is $(n-1)/4$ or $(n-5)/4.$ Hence, the remaining vertices in the subtournament are in-neighbours of $x$, so $x$ and $y$ have $(n-1)/4$ or $(n-5)/4$ common in-neighbours. It follows that 
$$|N^-(x)\triangle N^-(y)|\geq 2\left(\frac{n-1}{2}\right) - 2\left(\frac{n-1}{4}\right)=\frac{n-1}{2}.$$ 
By Lemma \ref{distset}, for $n\geq 5,$ $D$ has a distinguishing set of cardinality at most $\ceil{4n\log n/(n-1)}.$
\end{proof}

The second class we consider are team tournaments, i.e.\ orientations of the complete multipartite graph.  In particular, an $(m,r)$\emph{-team tournament} is an orientation of the complete multipartite graph with $m$ parts of size $r$. To avoid vacuous cases we will assume that $m>1$ and $r>1$.  If $D$ is $(m,r)$-team tournament $D$ such that $D$ is $k$-regular and the number of directed paths of length $2$ from $x$ to $y$ is $\alpha$ if there is an arc from $x$ to $y$, $\beta$ if there is an arc from $y$ to $x$, or $\gamma$ if there is no arc between $x$ and $y$, then we say that $D$ is a \emph{doubly regular $(m, r)$-team tournament}. 

In \cite[{\S}4.1]{drtt}, J{\o}rgensen et al.\ proved that any doubly regular $(m, r)$-team tournament is one of three types. Let $V_1, V_2, \ldots, V_m$ denote the partition into $m$ independent sets of order $r$. The three types are as follows.

\begin{itemize}
    \item Type I: $\mathcal{C}_r(T)$ for some doubly regular tournament $T$ of order $m$, where $\mathcal{C}_r(T)$ denotes the $(m, r)$-team tournament obtained from $T$ via an operation they call ``coclique extension''. 
    \item Type II: Every vertex in $V_i$ dominates exactly half of the vertices in each $V_j$, for $j \neq i$. 
    \item Type III: Every vertex in $V_i$ dominates either all vertices of $V_j$, exactly half of the vertices
in each $V_j$, or none of the vertices of $V_j$, for $j \neq i$, but is neither type I nor type II.
\end{itemize}

See Figure \ref{drmrtt} for a doubly regular $(4,2)$-team tournament of type II with $\alpha=\beta=1$ and $\gamma=3$.
\newpage
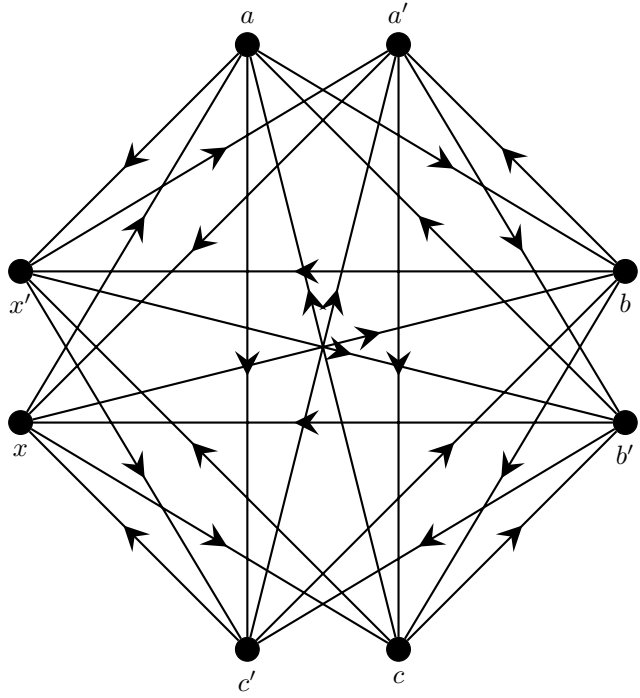
\begin{figure}[h] 
\centering \begin{tikzpicture}[scale=1]

 \def \n {5}
 
\node[draw,scale=0.95, circle, fill=black,label=above:{$a$}] (a) at (-1,4) {};
\node[draw,scale=0.95, circle, fill=black,label=above:{$a'$}] (a') at (1,4) {};

\node[draw,scale=0.95, circle, fill=black,label=below:{$b$}] (b) at (4,1) {};
\node[draw,scale=0.95, circle, fill=black,label=below:{$b'$}] (b') at (4,-1) {};
\node[draw,scale=0.95, circle, fill=black,label=below:{$c$}] (c) at (1,-4) {};
\node[draw,scale=0.95, circle, fill=black,label=below:{$c'$}] (c') at (-1,-4) {};

\node[draw,scale=0.95, circle, fill=black,label=below:{$x$}] (x) at (-4,-1) {};
\node[draw,scale=0.95, circle, fill=black,label=below:{$x'$}] (x') at (-4,1) {};
  
\path [fixed point arithmetic, thick, draw=black,
         decorate,
         postaction={
    on each segment={decoration={markings,mark=at position 0.55 with {\arrow[line width =3pt]{stealth}}},decorate}}]
(x) -- (a)
(x) -- (c)
(a)--(b)
(a)--(x')
(a)--(c')
(b)--(c)
(b)--(x')
(b)--(a')
(c)--(x')
(c)--(b')
(x') -- (a')
(x') -- (b')
(x') -- (c')
(a')--(x)
(a')--(c)
(a')--(b')
(b')--(x)
(b')--(a)
(b')--(c')
(c')--(x)
(c')--(b);
\draw[fixed point arithmetic, thick,
         decorate,
         postaction={
    on each segment={decoration={markings,mark=at position 0.6 with {\arrow[line width =3pt]{stealth}}},decorate}}]
(c) -- (a)
(c')--(a')
(x)--(b);

\end{tikzpicture}
\caption{A doubly regular $(4,2)$-team tournament of type II.}
\label{drmrtt}
\end{figure}

Let $d_i(x)$ denote the number of out-neighbours in $V_i$ of a vertex $x$. Any doubly regular $(m, r)$-team tournament of type II has the following properties.

\begin{lemma}{\em\cite[Theorem 4.3]{drtt}}\label{drttdegrees}
Let $D$ be a doubly regular $(m,r)$-team tournament of type II. Then $\alpha=\beta$, $r$ is even, and $\displaystyle{d_i(x)=\frac{r}{2}}$ for all $x\not\in V_i$.
\end{lemma}

\begin{lemma}{\em\cite[Theorem 4.4]{drtt}}\label{alphabeta}
Let $D$ be a doubly regular $(m,r)$-team tournament of type II. Then
\begin{enumerate}
    \item $\displaystyle{\alpha=\beta=\frac{r(m-2)}{4}}$,
    \item $\displaystyle{\gamma=\frac{(m-1)r^2}{4(r-1)}}$, and
    \item $r-1$ divides $m-1$.
\end{enumerate}
\end{lemma}

In particular, we note that part 3 of Lemma~\ref{alphabeta} implies that $r\leq m$.  Next, we make the following observation.

\begin{lemma} \label{DRTT_ANRD}
Let $D$ be a doubly regular $(m,r)$-team tournament of type II. Then $D$ is an asymmetric normally regular digraph with parameters $n=mr$, $\displaystyle{k=\frac{r(m-1)}{2}}$, $\displaystyle{\lambda=\frac{r(m-2)}{4}}$ and \linebreak $\displaystyle{\mu=\frac{r(r-2)(m-1)}{4(r-1)}}$.
\end{lemma}

\begin{proof}
The fact that $n=mr$ is clear from the definition, while $k=r(m-1)/2$ follows immediately from Lemma~\ref{drttdegrees}.

Now suppose that the vertices of $D$ are partitioned into independent sets $V_1, V_2, \ldots, V_m$.  First, let $x$ and $y$ be a pair of vertices with an arc from $x$ to $y$: since $D$ is an orientation of a complete multipartite graph it follows that $x\in V_i$ and $y\in V_j$ for some $i\neq j$.  The number of common out-neighbours (and hence, the number of common in-neighbours) of $x$ and $y$ is $k-d_j(x)-\alpha=k-d_i(y)-\beta$ (see \cite[Lemma 4.1]{drtt}).  
From Lemma~\ref{drttdegrees} we know that $d_i(v)=r/2$ for any $v\not\in V_i$. It follows from Lemma \ref{alphabeta} that
\begin{eqnarray*}
\lambda = |N^-(x)\cap N^-(y)| & = & \frac{r(m-1)}{2}-\frac{r}{2}-\alpha \\
                              & = & \frac{r(m-2)}{4}.
\end{eqnarray*}

Finally, let $x$ and $y$ be a pair of non-adjacent vertices, which means that $x,y\in V_i$ for some $i$.  There are $\gamma$ paths of length $2$ from $x$ to $y$. Hence, of the $r(m-1)/2$ in-neighbours of $y$, $\gamma$ are out-neighbours of $x$ and $r(m-1)/2-\gamma$ are in-neighbours of $x$. Again, it follows from Lemma \ref{alphabeta} that 
\begin{eqnarray*}
\mu = |N^-(x)\cap N^-(y)| & = & \frac{r(m-1)}{2}-\gamma \\
                          & = & \frac{r(m-1)}{2} - \frac{r^2(m-1)}{4(r-1)} \,\,\textnormal{ (by Lemma \ref{alphabeta})} \\
                          & = & (m-1) \left( \frac{r}{2}-\frac{r^2}{4(r-1)} \right) \\
                          & = & (m-1) \left(\frac{2r(r-1) - r^2}{4(r-1)}\right) \\
                          & = & \frac{r(r-2)(m-1)}{4(r-1)}.
\end{eqnarray*}
This completes the proof.
\end{proof}

Since these are examples of ANRDs, we can immediately apply Theorem~\ref{ANRD} which yields an upper bound of $O(\sqrt{n}\log n)$ for the localization number and metric dimension of these digraphs.  However, we can apply Lemma~\ref{distset} directly to obtain a tighter bound.

\begin{theorem} \label{DRTT}
Let $D$ be a doubly regular $(m,r)$-team tournament of type II. If $n=mr\geq 27$ then $\zeta(D) \leq \dim(D) \leq \ceil{4\log n}$.
\end{theorem}

\begin{proof} 
To apply Lemma~\ref{distset}, we need to determine $\max\{\lambda,\mu\}$. Looking at the difference of the two values obtained in Lemma~\ref{DRTT_ANRD}, we see that 
\begin{eqnarray*}
\lambda-\mu & = & \frac{r(m-2)}{4}  - \frac{r(r-2)(m-1)}{4(r-1)}\\
            & = & \frac{r(r-1)(m-2) - r(r-2)(m-1)}{4(r-1)}\\
            & = & \frac{mr-r^2}{4(r-1)} \\
            & = & \frac{r}{4(r-1)}(m-r) \\
            & \geq & 0 \phantom{\int}
\end{eqnarray*}
since $m\geq r$ and $r>1$.  Therefore $\lambda\geq\mu$.  It follows that for any pair of vertices $x,y$, 
$$|N^-(x)\triangle N^-(y)|\geq 2\frac{r(m-1)}{2}-2 \frac{r(m-2)}{4}=\frac{mr}{2}=\frac{n}{2}.$$
Since $\frac{n}{2}>4\log n$ when $n\geq 27$, we can apply Lemma \ref{distset} to find a distinguishing set of cardinality at most 
$$\left\lceil \frac{2n\log n}{n/2} \right\rceil = \ceil{4\log n}.$$ 
\end{proof}

We conclude this section by observing that, in common with doubly regular tournaments, both nearly doubly regular tournaments and doubly regular team tournaments of type II of order $n$ have upper bounds of $O(\log n)$ on their localization number and metric dimension.


\section{Future directions}
A natural open question is identify other classes of highly structured digraphs where the $O(\sqrt{n}\log n)$ upper bound (see Theorem~\ref{ANRD}) on localization number and metric dimension still applies -- for instance, does it hold for normally regular digraphs with digons, for Deza digraphs, or for divisible design digraphs?  We also saw that this bound can be improved to $O(\log n)$ in some special cases, such doubly regular $(m,r)$-team tournaments of type II, which leads us to ask if the $O(\sqrt{n}\log n)$ bound is tight.

A more challenging question would be to investigate directed analogues of strongly regular graphs that are defined in terms of the number of directed paths of length $2$ between vertex pairs, such as directed strongly regular graphs. Unlike the common-neighbour variants we consider here, these digraphs do not necessarily have fixed numbers of common in-neighbours, which prevents a direct application of our existing methods. The central open problem would be to determine whether similar probabilistic or combinatorial methods can be used to obtain nontrivial upper bounds on the localization number and metric dimension of these digraphs. 


\section*{Acknowledgements}
The first author is supported by an NSERC Discovery Grant.  The second author is supported by an NSERC Postdoctoral Fellowship.


\end{document}